\documentclass[12pt]{article}
\usepackage{amsmath}
\usepackage{amssymb}
\usepackage{amsfonts}
\usepackage{amsmath,amsthm,amssymb}
\allowdisplaybreaks

\newcommand{\ds}{\displaystyle}

\theoremstyle{plain}

\def \[{\begin{equation}}
\def \]{\end{equation}}

\newtheorem{thm}{Theorem}[section]

\newtheorem{lem}[thm]{Lemma}

\newtheorem{rem}[thm]{Remark}

\newif \ifLastSection \LastSectionfalse

\numberwithin{equation}{section}

\newcommand{\R}{{\mathbb R}}

\begin{document}
\title{   Existence of positive solutions for nonlinear Kirchhoff type problems in
$\R^{3}$ with critical Sobolev exponent and sign-changing nonlinearities
\thanks {a: Partially  supported  by  NSFC No: 11071095 and  Hubei Key Laboratory of Mathematical
Sciences.  }
\thanks{b:Corresponding author: Gongbao Li, School of Mathematics and
Statistics, Central China  Normal University,  Wuhan, 430079, China.
E-mail address: ligb@mail.ccnu.edu.cn}}
\author{Gongbao Li ,\,\,\,\,\,Hongyu Ye \\ \small {School of Mathematics and Statistics,  Central China Normal
University,}
\\
\small{ Wuhan 430079, P. R. China}}
\date{}
\maketitle

\begin{abstract}
In this paper, we study the following nonlinear problem of Kirchhoff
type with critical Sobolev exponent:
$$
\left\{%
\begin{array}{ll}
    -\left(a+b\ds\int_{\R^3}|D u|^2\right)\Delta u+u=f(x,u)+u^{5}, & \hbox{$x\in \R^3$}, \\
    u\in H^1(\R^3),~~~~u>0, & \hbox{$x\in \R^3$},\\
\end{array}%
\right.$$where $a,$ $b>0$ are constants.
Under certain assumptions on the sign-changing function $f(x,u)$, we prove
the existence of positive solutions by variational methods.

Our main results can be viewed as a partial extension of a recent
result of He and Zou in \cite{hz} concerning the existence of
positive solutions to the nonlinear Kirchhoff problem
$$\left\{%
\begin{array}{ll}
    -\left(\varepsilon^2a+\varepsilon b\ds\int_{\R^3}|D u|^2\right)\Delta u+V(x)u=f(u), & \hbox{$x\in \R^3$}, \\
    u\in H^1(\R^3),~~~~u>0, & \hbox{$x\in \R^3$},\\
\end{array}%
\right.$$where $\varepsilon>0$ is a parameter, $V(x)$ is a positive
continuous potential and $f(u)\thicksim |u|^{p-2}u$ with $4<p<6$ and
satisfies the Ambrosetti-Rabinowitz type condition.
\end{abstract}

\section{ Introduction and main result}
In this paper, we consider the existence of positive solutions to
the following nonlinear Kirchhoff type problem with critical Sobolev
exponent:
\begin{equation}\label{1.1}
\left\{%
\begin{array}{ll}
    -\left(a+b\ds\int_{\R^3}|D u|^2\right)\Delta u+u=f(x,u)+u^{5}, & \hbox{$x\in \R^3$}, \\
    u\in H^1(\R^3),~~~~u>0, & \hbox{$x\in \R^3$},\\
\end{array}%
\right.\end{equation} where $a,$ $b>0$ are two constants and
$f(x,t):\R^3\times\R\rightarrow\R$ satisfies the following two kinds of
conditions:

\noindent(A)~~$f(x,t)=f(t)$. $f\in C(\R)$ satisfies:

$(f_1)$~~~$\lim\limits_{|t|\rightarrow0}\frac{f(t)}{t^3}=0$ and $f(t)\equiv0$ for all
$t\leq0$;

$(f_2)$~~~$\frac{f(t)}{t^3}$ is strictly increasing for $t>0$;

$(f_3)$~~~$\lim\limits_{|t|\rightarrow+\infty}\frac{f(t)}{t^5}=0$;

$(f_4)$~~~there exists a $\mu>4$ such that $$0<\mu
F(t)\leq f(t)t~~\hbox{for~all}~t>0,$$ where $F(t)=\int_{0}^tf(s)ds$.

\noindent(B)~~$f(x,t)=f_\lambda(x)|t|^{p-2}t$, where $2\leq p<4$,
$\lambda>0$ is a parameter,
$$f_\lambda(x)=\lambda f^+(x)-f^-(x)$$
is a
sign-changing weight function and
$$\Sigma\triangleq\{x\in\R^3|~f^+(x)\neq0\}~~\hbox{is~a~nonempty~domain},$$
where $f\in
L^{p_*}(\R^3),$ $p_*=\frac{6}{6-p}$ and $f^\pm(x)=\max\{\pm f(x),0\}$.

In recent years, the following elliptic problem
\begin{equation}\label{1.2}
-\left(a+b\ds\int_{\R^N}|D u|^2\right)\Delta u+V(x)u=f(x,u),~~u\in H^1(\R^N) \\
\end{equation} has been extensively studied by many
researchers, where $V:\R^N\rightarrow\R$, $f\in C(\R^N\times\R,\R)$,
$N=1,2,3$ and $a,b>0$ are constants. \eqref{1.2} is a nonlocal
problem as the appearance of the term  $\int_{\R^N}|D u|^2$ implies
that \eqref{1.2} is not a pointwise identity. This causes some
mathematical difficulties which make the study of \eqref{1.2}
particularly interesting. Problem \eqref{1.2} arises in an
interesting physical context. Indeed, if we set $V(x)=0$ and replace
$\R^N$ by a bounded domain $\Omega\subset\R^N$ in \eqref{1.2}, then
we get the following Kirchhoff Dirichlet problem
\begin{equation}\label{1.3}\left\{%
\begin{array}{ll}
    -\left(a+b\ds\int_{\Omega}|D u|^2\right)\Delta u=f(x,u), & \hbox{$x\in\Omega$}, \\
    u=0, & \hbox{$x\in \partial\Omega$},\\
\end{array}%
\right.\end{equation} which is related to the stationary analogue of
the equation
$$
\rho\frac{\partial^2u}{\partial
t^2}-\left(\frac{P_0}{h}+\frac{E}{2L}\ds\int_0^L\left|\frac{\partial
u}{\partial x}\right|^2dx\right)\frac{\partial^2 u}{\partial
x^2}=0$$ presented by Kirchhoff in \cite{k}. The readers can learn
some early research of Kirchhoff equations from \cite{b,p}. In
\cite{l}, J.L. Lions introduced an abstract functional analysis
framework to the following equation
\begin{equation}\label{1.4}
u_{tt}-\left(a+b\ds\int_{\Omega}|D u|^2\right)\Delta
u=f(x,u).\end{equation}After that, \eqref{1.4} received much
attention, see \cite{ac,af,ap1,cd,ds} and the references therein.

Before we review some results about \eqref{1.2}, we give several notations and
definitions.

Throughout this paper, we use standard notations. For simplicity, we
write $\int_{\Omega} h$ to mean the Lebesgue integral of $h(x)$ over
a domain $\Omega\subset\R^3$. $L^{p}\triangleq L^{p}(\R^{3})~(1\leq
p<+\infty)$ is the usual Lebesgue space with the standard norm
$|u|_{p}.$ We use $\rightarrow$ and $\rightharpoonup$ to denote the
strong and weak convergence in the related function space
respectively. For $x\in\R^3$, $B_r(x)\triangleq\{y\in\R^3|~|x-y|<r\}$. $C$ will
denote a positive constant unless specified. We denote a subsequence
of a sequence $\{u_n\}$ as $\{u_n\}$ to simplify the notation unless
specified.

Let $(X,\parallel\cdot\parallel)$ be a Banach space with its dual
space $(X^*,\parallel\cdot\parallel_*)$, $I\in C^1(X,\R)$ and
$c\in\R$. We say a sequence $\{x_n\}$ in $X$ a Palais-Smale sequence
at level $c$ ($(PS)_{c}$ sequence in short) if $I(x_{n})\rightarrow
c$ and $\parallel I'(x_{n})\parallel_*\rightarrow 0$ as
$n\rightarrow\infty.$ We say that $I$ satisfies $(PS)_c$ condition
if for any $(PS)_c$ sequence $\{x_n\}$ in $X$, there exists a
subsequence $\{x_{n_k}\}$ such that $x_{n_k}\rightarrow x_0$ in $X$
for some $x_0\in X$.

There have been many works about the existence of nontrivial
solutions to \eqref{1.2} with subcritical nonlinearities by using
variational methods, see e.g. \cite{hz,jw,lls,lh,w}. A typical way
to deal with \eqref{1.2} with subcritical nonlinearities is to use
the mountain-pass theorem. For this purpose, one usually assumes
that $f(x,u)$ is subcritical, superlinear at the origin and either
4-superlinear at infinity in the sense that
$\lim\limits_{|u|\rightarrow+\infty}\frac{F(x,u)}{u^4}=+\infty~$
uniformly in $x\in\R^N$ or satisfies the Ambrosetti-Rabinowitz type
condition ($(AR)$ in short):
$$(AR)~~~~~~~~~~~~\exists~\mu>4~~\hbox{such~that}~~~~0<\mu~F(x,u)\leq f(x,u)u~~~~\hbox{for~all~}u\neq0,~~~~~~~~~~~~~~~~~~~~$$
where $F(x,u)=\int_0^u f(x,t)dt$. Under these mentioned conditions,
one easily sees that the functional corresponding to \eqref{1.2}
possesses a mountain-pass geometry around $0\in H^1(\R^N)$ and then
by the mountain-pass theorem, one can get a $(PS)$ sequence.
Moreover, the $(PS)$ sequence is bounded in $H^1(\R^N)$ if $$(F)~~~~~~~~~~~~~~~~~~~~~~~~~~~~~~~~~~~~4F(x,u)\leq f(x,u)u~~\hbox{
for~all}~u\in\R.~~~~~~~~~~~~~~~~~~~~~~~~~~~~~~~~~~~~~~~~~$$ Therefore, one can show that \eqref{1.2} has
at least one nontrivial solution provided some further conditions on
$f(x,u)$ and $V(x)$ are assumed to guarantee that the $(PS)$
condition holds.

After Brezis and Nirenberg in \cite{bn} first studied a critical
growth problem in a bounded domain:
$$\left\{%
\begin{array}{ll}
    -\Delta u=|u|^{2^*-2}u+\lambda u, & \hbox{$x\in\Omega$}, \\
    u=0, & \hbox{$x\in \partial\Omega$},\\
\end{array}%
\right.
$$
many researchers have considered
problems with critical Sobolev exponent by either pulling the energy level down below some
critical energy level to recover certain
compactness or using a combination of the idea above with the
concentration compactness principle of P.L. Lions \cite{pl} and proved the existence of nontrivial solutions. An interesting
question now is whether the same existence results occur to the
nonlocal problem \eqref{1.2} with critical Sobolev exponent.
To the best of our knowledge, there are few papers about the
non-existence and existence results to \eqref{1.2} when $f(x,u)$
exhibits a critical growth. In fact, by the Poho\u{z}ave identity, the following problem
\begin{equation}\label{1.5}
-\left(a+b\ds\int_{\R^3}|Du|^2\right)\Delta u+u=u^5,~~u\in
H^1(\R^3),
\end{equation}
has no
nontrivial solution (see Theorem 1.1 below). Therefore, in the spirit of \cite{bn}, one usually adds a lower order
perturbation to the right-hand side of equation \eqref{1.5}. In
\cite{af2}, Alves and Figueiredo studied the following class of
Kirchhoff problem
\begin{equation}\label{1.6}
M\left(\ds\int_{\R^N}(|Du|^2+V(x)u^2)\right)[-\Delta
u+V(x)u]=\lambda f(u)+ \gamma u^\tau~~\hbox{in}~\R^N,
\end{equation}
where $\tau=5$ for $N=3$ and $\tau\in(1,+\infty)$ for $N=1,2$;
$\lambda>0$ and $\gamma\in\{0,1\}$, i.e. they considered both subcritical
and critical cases. When $\gamma=1$, they proved under certain
conditions on functions $M(t)$, $V(x)$ and $f(u)$, then there exists a
$\lambda_*>0$ such that problem \eqref{1.6} has at least a positive solution
for all $\lambda\geq\lambda_*$. Recently, in \cite{wt}, Wang et al. proved the existence and multiplicity of positive ground state
solutions for the following Kirchhoff problem with critical growth
\begin{equation}\label{1.7}
\left\{%
\begin{array}{ll}
    -\left(\varepsilon^2a+\varepsilon b\ds\int_{\R^3}|Du|^2\right)\Delta u+V(x)u=\lambda f(u)+u^5, & \hbox{$x\in \R^3$}, \\
    u\in H^1(\R^3),~~~~u>0, & \hbox{$x\in \R^3$},\\
\end{array}%
\right. \end{equation} where $V(x)$ satisfies that
$V_\infty=\liminf\limits_{|x|\rightarrow\infty}V(x)>V_0=\inf\limits_{x\in\R^3}V(x)>0$ and
the nonlinearity $f\in C(\R^3)$ satisfies $(f_1)-(f_3)$. By using the
Nehari manifold and pulling the energy level down below the
following critical level:
\begin{equation}\label{1.14}
c_1\triangleq\frac{1}{3}(aS)^{\frac{3}{2}}+\frac{1}{12}b^3S^6,
\end{equation}
where $S=\inf\limits_{D^{1,2}(\R^3)\backslash\{0\}}\frac{\int_{\R^3}|Du|^2}{|u|_6^2}$,
they showed that there exist $\varepsilon^*>0$, $\lambda^*>0$ such
that for any $\varepsilon\in (0,\varepsilon^*)$ and
$\lambda\in[\lambda^*,\infty)$, problem \eqref{1.7} has at least one positive
ground state solution in $H^1(\R^3)$.

Motivated by the works described before, in this paper, we try
to prove the existence of positive solutions for problem \eqref{1.1} with $f(x,t)$ satisfying cases (A) and (B) respectively. To state our main results, for fixed $a>0$, we introduce an
equivalent norm on $H^1(\R^3)$: the norm of $u\in H^1(\R^3)$ is
defined as
$$\|u\|=\left(\ds\int_{\R^3}a|Du|^2+|u|^2\right)^{\frac{1}{2}},$$which
is induced by the corresponding inner product on $H^1(\R^3)$. Clearly, weak solutions to \eqref{1.1} correspond to critical points
of the following functionals
\begin{equation}\label{1.8}
I(u)=\frac{1}{2}\ds\int_{\R^3}(a|Du|^2+|u|^2)+\frac{b}{4}\left(\ds\int_{\R^3}|Du|^2\right)^2-\ds\int_{\R^3}F(u)-\frac{1}{6}\ds\int_{\R^3}|u|^{6}
\end{equation}
and
\begin{equation}\label{1.9}
I_{f_\lambda}(u)=\frac{1}{2}\ds\int_{\R^3}(a|Du|^2+|u|^2)+\frac{b}{4}\left(\ds\int_{\R^3}|Du|^2\right)^2-\frac{1}{p}\ds\int_{\R^3}f_\lambda(x)|u|^{p}-\frac{1}{6}\ds\int_{\R^3}|u|^{6}.
\end{equation}
 We see that
$I(u),~I_{f_\lambda}(u)\in C^1(H^1(\R^3),\R)$. We
say a nontrivial solution $u\in H^1(\R^3)$ to \eqref{1.1} a ground
state solution if $I(u)\leq I(w)$ for any nontrivial solution $w\in H^1(\R^3)$ to \eqref{1.1}. \\

Based on the Poho\u{z}ave identity, we have the following
non-existence result:

\begin{thm}\label{thm1.1}\ \
For any $a,$ $b>0$, problem \eqref{1.5} has no nontrivial solution in $H^1(\R^3)$.
\end{thm}
\begin{rem}\label{rem1.2}\ \
We can extend Theorem 1.1 to the
non-constant potential case, i.e. the following problem
$$
-\left(a+b\ds\int_{\R^3}|Du|^2\right)\Delta u+V(x)u=u^5,~~u\in
H^1(\R^3)
$$
has no nontrivial solution in $H^1(\R^3)$ if $V(x)\in C^1(\R^3,\R)$,
there exist $C_1,~C_2>0$ such that $0< C_1\leq V(x)\leq C_2$ for all
$x\in\R^3$ and
$$2V(x)+(DV(x),x)\geq0~~\hbox{for}~a.e.~x\in\R^3.$$
\end{rem}

Our main results are as follows:
\begin{thm}\label{thm1.3}\ \
Suppose that (A) holds, then \eqref{1.1} has at least a positive ground state solution in $H^1(\R^3)$.
\end{thm}

\begin{thm}\label{thm1.4}\ \
Suppose that (B) holds for $p\in[2,4)$, then there
exists $\lambda^*>0$ such that for all $\lambda\in(0,\lambda^*)$, \eqref{1.1} has at least a positive solution in $H^1(\R^3)$.
\end{thm}

Theorem 1.3 can be viewed as a partial extension of a main result in
\cite{hz}.

Now we give our main ideas for the proof of Theorem 1.3 and 1.4.
We prove Theorem 1.3 by showing that $I$ has a positive ground state
critical point in $H^1(\R^3)$. By $(f_1)(f_3)(f_4)$, we can get a
bounded $(PS)_c$ sequence of $I$, however, it is not easy to see that
$I^\prime$ is weakly continuous by direct calculations since equation
\eqref{1.1} is no longer a pointwise identity. Indeed, in general, we do not know $\int_{\R^3}|Du_n|^2\rightarrow\int_{\R^3}|Du|^2$ from $u_n\rightharpoonup u$ in $H^1(\R^3)$. We succeed in doing
so by using the method used in \cite{hz}, which strongly relies on the condition $(f_2)$. Hence, there exists a
critical point for $I$. As we deal
with the critical problem \eqref{1.1} in $H^1(\R^3)$, the Sobolev embeddings
$H^1(\R^3)\hookrightarrow L^q(\R^3),$ $q\in[2,6)$ are not compact.
The functional $I$ does not satisfy $(PS)_c$ condition at every energy
level $c$. To overcome this difficulty, we try to pull the energy level down below some
critical level $c^*$. Considering the nonlocal effect, it is more complicated to handle and careful analysis is needed. $c^*$ is given as follows:
\begin{equation}\label{1.10}
c^*\triangleq\frac{abS^3}{4}+\frac{b^3S^6}{24}+\frac{(b^2S^4+4aS)^{\frac{3}{2}}}{24},
\end{equation}
which is larger than $c_1$ given in \eqref{1.14}. Then we apply the concentration compactness principle to prove that $I$ satisfies $(PS)_c$ condition for any $c\in(0,c^*)$, which implies that $I$ has a nontrivial critical point. Whence a nontrivial
critical point for $I$ has been obtained, the existence of a ground
state critical point follows by standard argument. Then the
proof of Theorem 1.3 is completed.

We will prove Theorem 1.4 by using the mountain pass theorem. To do
so, we try to get a $(PS)_{c_\lambda}$ sequence and to prove
that the $(PS)_{c_\lambda}$ sequence is bounded in $H^1(\R^3)$ and converges to
a positive critical point of $I_{f_\lambda}$ in $H^1(\R^3)$, where $c_\lambda$ is a mountain-pass level. There are some
difficulties. First, since $p\in[2,4)$, condition $(F)$ does not hold, let alone $(AR)$. It is not easy to get the
boundedness of the $(PS)_{c_\lambda}$ sequence. We succeed in doing so by using conditions imposed on $f^+(x)$. Secondly, $p\in[2,4)$ implies that the monotonicity of $\frac{f_\lambda(x)|u|^{p-2}u}{u^3}$ is not true. Then the method to prove that $I^\prime$ is weakly continuous, which was used in the proof of Theorem 1.3, can not be applied here. To overcome this difficulty, although we can not directly prove that the weak limit $u\in H^1(\R^3)$ of a $(PS)_{c_\lambda}$ sequence $\{u_n\}$ is a critical point of $I_{f_\lambda}$, but we do easily see that $u$ is a critical point of the following functional
$$J_{f_\lambda}(u)=\frac{a+bA^2}{2}\ds\int_{\R^3}|Du|^2+\frac{1}{2}\ds\int_{\R^3}|u|^2-\frac{\lambda}{p}\ds\int_{\R^3}f_\lambda(x)|u|^{p}-\frac{1}{6}\ds\int_{\R^3}|u|^6$$
and $\{u_n\}$ is a $(PS)_{c_\lambda+\frac{bA^4}{4}}$ sequence for $J_{f_\lambda}$,
where $A^2=\lim\limits_{n\rightarrow\infty}\int_{\R^3}|Du_n|^2.$
We try to prove that $I_{f_\lambda}$ satisfies $(PS)_{c_\lambda}$ condition with the help of $J_{f_\lambda}$ and by pulling the mountain-pass level $c_\lambda$ down below some critical energy level $c_2$.
As $p\in[2,4)$ and $f_\lambda(x)$
is a sign-changing weight function, it is difficult to get the
critical energy level $c_2$. Inspired by \cite{bn2}, we succeed in obtaining $c_2$ by choosing a suitable cut off function and a suitable
test function. Indeed, since $\Sigma$ is a nonempty domain, we may assume that $0\in\Sigma$ and $B_{R_0}(0)\subset\Sigma$ for some $R_0>0$. For any $\varepsilon>0$, we consider the
following test function
$$w_{\varepsilon}(x)=\eta(x)\widetilde{U}_{\varepsilon}(x),~~x\in\R^3,$$
where $\widetilde{U}_{\varepsilon}(x)=\frac{1}{(\varepsilon^2+|x|^2)^{\frac{1}{2}}}$ and $\eta\in C_0^\infty(\Sigma)$ with
$\eta\geq0$ and $\eta|_{B_{R_0}(0)}\equiv1$. By careful analysis, we proved that there
exists a $\lambda^*>0$ such that for $\lambda\in(0,\lambda^*)$,
the critical energy level is as follows:
\begin{equation}\label{1.11}
c_2\triangleq c^*-\frac{4-p}{4}\left(\frac{3p}{b}\right)^{\frac{p}{4-p}}\left(\frac{6-p}{6p}\frac{|f^+|_{p_*}}{S^{\frac{p}{2}}}\right)^{\frac{4}{4-p}}\lambda^{\frac{4}{4-p}}>0,
\end{equation}
where $c^*$ is given in \eqref{1.10}. Then Theorem 1.4 is proved.

The paper is organized as follows. In $\S$ 2, we prove Theorem 1.1.
In $\S$ 3, we present some preliminary results which will be used to
prove Theorems 1.3 and 1.4. In $\S$ 4, we will prove our main
results Theorems 1.3 and 1.4.

\section{Proof of Theorem 1.1 }
\begin{lem}\label{lem2.1}(Poho\u{z}ave Identity)\ \  Let $u\in H^1(\R^3)$ be a weak solution to
problem \eqref{1.5}, then we have the following Poho\u{z}aev
identity:
\begin{equation}\label{2.1}
\frac{a}{2}\ds\int_{\R^3}|Du|^2+\frac{3}{2}\ds\int_{\R^3}|u|^2+\frac{b}{2}\left(\ds\int_{\R^3}|Du|^2\right)^2-\frac{1}{2}\ds\int_{\R^3}|u|^{6}=0.
\end{equation}
\end{lem}
\begin{proof}~~
Let $M^2\triangleq\int_{\R^3}|Du|^2\in\R$, then $u\in H^1(\R^3)$ is a weak solution to the following problem
$$-\triangle u=\frac{1}{a+bM^2}(u^5-u).$$
Then by a standard argument, we conclude that (2.1) holds.
\end{proof}

\noindent $\textbf{Proof of Theorem 1.1}$\,\,\
\begin{proof}~~
Suppose that $u\in H^1(\R^3)$ is a solution to problem
\eqref{1.5}, then $u$ satisfies the Poho\u{z}aev identity
\eqref{2.1} and
$$
a\ds\int_{\R^3}|Du|^2+\ds\int_{\R^3}|u|^2+b\left(\ds\int_{\R^3}|Du|^2\right)^2-\ds\int_{\R^3}|u|^{6}=0.
$$
Hence we conclude that $\int_{\R^3}|u|^2=0,$ which implies that
$u=0$.
\end{proof}

\section{Preliminary results for Theorem 1.3 and 1.4 }
In this section, we will give some preliminary results which will be
used in the proof of Theorems 1.3 and 1.4.

Throughout this paper, for each $q\in[2,6]$, by the Sobolev embeddings, we denote
\begin{equation}\label{3.25}
S_q=\inf\limits_{H^1(\R^3)\backslash\{0\}}\frac{\|u\|^2}{|u|_q^2}>0~~~~\hbox{and}
~~~~S=\inf\limits_{D^{1,2}(\R^3)\backslash\{0\}}\frac{\int_{\R^3}|Du|^2}{|u|_6^2}>0.
\end{equation}
Then
\begin{equation}\label{3.1}
|u|_q\leq
S_q^{-\frac{1}{2}}\|u\|~~\hbox{for}~u\in H^1(\R^3)
\end{equation}
and
\begin{equation}\label{3.2}
|u|_6\leq
S^{-\frac{1}{2}}\left(\ds\int_{\R^3}|Du|^2\right)^{\frac{1}{2}}~~\hbox{for}~u\in
D^{1,2}(\R^3).
\end{equation}
\begin{lem}\label{lem3.1}\ \

\noindent(i)~~~Suppose that (A) holds, then the functional $I$ possesses a
mountain pass geometry:

(a)~~There exist $\alpha,$ $\rho>0$ such that $I(u)\geq\alpha$ for
all $\|u\|=\rho$.

(b)~~There exists an $e\in B_\rho^c(0)$ such that $I(e)<0$.

\noindent(ii)~~~Suppose that (B) holds for $p\in[2,4)$, let $\lambda_0=\frac{S_6}{|f^+|_{\frac{3}{2}}}$, then the
functional $I_{f_\lambda}$ possesses the mountain-pass geometry for
all $\lambda>0$ if $p\in(2,4)$ or all $\lambda\in(0,\lambda_0)$
if $p=2$.
\end{lem}
\begin{proof}~~

\noindent(i)~~~(a)~~By $(f_1)(f_3)$, for any $\varepsilon>0$,
there exists $C_\varepsilon>0$ such that $F(u)\leq \varepsilon|u|^4+C_\varepsilon u^6.$ Then by \eqref{3.1}, we have that
$$\begin{array}{ll}
I(u)&\geq\frac{1}{2}\|u\|^2+\frac{b}{4}\left(\ds\int_{\R^3}|Du|^2\right)^2-\varepsilon\ds\int_{\R^3}|u|^4-C_\varepsilon\ds\int_{\R^3}|u|^6\\[5mm]
&\geq\frac{1}{2}\|u\|^2-\varepsilon S_4^{-2}\|u\|^4-S^{-3}_6C_\varepsilon\|u\|^6,\\[5mm]
\end{array}$$
hence there exist $\alpha,~\rho>0$ such that $I(u)\geq\alpha$ for
all $\|u\|=\rho$.

(b)~~By $(f_4)$, there exists $C>0$ such that
\begin{equation}\label{3.26}
F(u)\geq C|u|^\mu~~~\hbox{for~all}~u\in\R.\end{equation} For any
$u\in H^1(\R^3)\backslash\{0\}$, $t>0$,
$$
I(tu)\leq\frac{t^2}{2}\|u\|^2+\frac{bt^4}{4}\left(\ds\int_{\R^3}|Du|^2\right)^2-Ct^\mu\ds\int_{\R^3}|u|^\mu-\frac{t^6}{6}|u|^6
\rightarrow-\infty~~\hbox{as}~~t\rightarrow+\infty.
$$
Then there exists $t_0>0$ large such that $I(t_0u)<0$ and
$\|t_0u\|>\rho$.

\noindent(ii)~~~(a)~~Since $p\in[2,4)$, by \eqref{3.1},
$$\begin{array}{ll}
I_{f_\lambda}(u)&\geq\frac{1}{2}\|u\|^2+\frac{b}{4}\left(\ds\int_{\R^3}|Du|^2\right)^2-\frac{\lambda}{p}\ds\int_{\R^3}f^+(x)|u|^p-\frac{1}{6}\ds\int_{\R^3}|u|^6\\[5mm]
&\geq\left\{%
\begin{array}{ll}
    \frac{1}{2}\|u\|^2-\frac{\lambda}{p}|f^+|_{p_*}S_6^{-\frac{p}{2}}\|u\|^p-\frac{S_6^{-3}}{6}\|u\|^6, & \hbox{if~}p\in(2,4), \\
    \frac{1}{2}\left(1-\lambda|f^+|_{\frac{3}{2}}S_6^{-1}\right)\|u\|^2-\frac{S_6^{-3}}{6}\|u\|^6, & \hbox{if~}p=2.\\
\end{array}%
\right.
\end{array}
$$
Then for all $\lambda\in(0,\lambda_0)$ if $p=2$ or any
$\lambda>0$ if $p\in(2,4)$,
$I_{f_\lambda}|_{B_\delta(0)}>\beta$ for some $\delta,~\beta>0$.

(b)~~The proof is similar to that of (i).
\end{proof}

By the mountain-pass theorem (see e.g. Theorem 2.10 in \cite{wi}), there exists a $(PS)_{c}$ sequence $\{u_n\}\subset
H^1(\R^3)$ and a $(PS)_{c_\lambda}$ sequence $\{\tilde{u}_n\}\subset
H^1(\R^3)$ such that
$$I(u_n)\rightarrow c,~~~~~~I^\prime(u_n)\rightarrow
0~~\hbox{in}~H^{-1}(\R^3)$$and
$$I_{f_\lambda}(\tilde{u}_n)\rightarrow c_\lambda,~~~~~~I_{f_\lambda}^\prime(\tilde{u}_n)\rightarrow
0~~\hbox{in}~H^{-1}(\R^3),$$ where
$c=\inf\limits_{\gamma\in\Gamma}\max\limits_{t\in[0,1]}I(\gamma(t))$,
$c_\lambda=\inf\limits_{\tilde{\gamma}\in\widetilde{\Gamma}}\max\limits_{t\in[0,1]}I_{f_\lambda}(\tilde{\gamma}(t))$
and
$$\Gamma=\{\gamma\in
C([0,1],H^1(\R^3))|~\gamma(0)=0,~I({\gamma}(1))<0\},$$
$$\widetilde{\Gamma}=\{\tilde{\gamma}\in
C([0,1],H^1(\R^3))|~\tilde{\gamma}(0)=0,~I_{f_\lambda}({\tilde{\gamma}}(1))<0\}.$$

\begin{lem}\label{lem3.2}\ \
Suppose that $\{u_n\}$ and $\{\tilde{u}_n\}$ are $(PS)_c$ and $(PS)_{c_\lambda}$ sequences for $I$ and $I_{f_\lambda}$ respectively, then

(i)~~~there exists a $u\in H^1(\R^3)$ such that $I^\prime(u)=0$. Moreover, if $u\not\equiv0$, then
$$\ds\int_{\R^3}|Du_n|^2\rightarrow\ds\int_{\R^3}|Du|^2.$$

(ii)~~~$\{\tilde{u}_n\}$ is bounded in $H^1(\R^3)$.
\end{lem}

\begin{proof}~~
(i)~~~The proof is similar to that of Lemmas 2.2 and 2.4 in \cite{hz}. For reader's convenience, we give a detailed proof.

By $(f_4)$, we see that
$$c+1+\|u_n\|\geq I(u_n)-\frac{1}{4}\langle I^\prime(u_n),u_n\rangle\geq \frac{1}{4}\|u_n\|^2,$$
then $\{u_n\}$ is bounded in $H^1(\R^3)$. Up to a subsequence, we
may assume that, there exists a $u\in H^1(\R^3)$ and $A\in\R$ such that
\begin{equation}\label{3.3}
u_n\rightharpoonup u~~\hbox{in}~H^1(\R^3),
\end{equation}
and
$$\ds\int_{\R^3}|Du_n|^2\rightarrow A^2.$$
If $u\equiv0$, then the proof is completed. If $u\not\equiv0$, then we see that
$$\ds\int_{\R^3}|Du|^2\leq A^2.$$
Just suppose that $\int_{\R^3}|Du|^2<A^2.$ By $I^\prime(u_n)\rightarrow0$ in $H^{-1}(\R^3)$ and \eqref{3.3}, we have that
$$\ds\int_{\R^3}(aDuD\varphi+u\varphi)+bA^2\ds\int_{\R^3}DuD\varphi-\ds\int_{\R^3}f(u)\varphi-\ds\int_{\R^3}u^5\varphi=0,~~\forall~\varphi\in H^1(\R^3).$$
Then $\langle I^\prime(u),u\rangle<0$. $(f_1)$ $(f_3)$ imply that $\langle I^\prime(tu),tu\rangle>0$ for small $t>0$. Hence there exists a $t_0\in(0,1)$ satisfying $\langle I^\prime(t_0u),t_0u\rangle=0$. Moreover, by $(f_2)$ $(f_4)$, we see that $I(t_0u)=\max\limits_{t\in[0,1]}I(tu)$ and $tu\in\Gamma$. We easily conclude from $(f_2)$ that $\frac{1}{4}f(s)s-F(s)$ is strictly increasing in $s>0$. So
$$\begin{array}{ll}
c&\leq I(t_0u)=I(t_0u)-\frac{1}{4}\langle I^\prime(t_0u),t_0u\rangle\\[5mm]
&=\frac{1}{4}t_0^2\int_{\R^3}(a|Du|^2+|u|^2)+\int_{\R^3}\left(\frac{1}{4}f(t_0u)t_0u-F(t_0u)\right)+\frac{1}{12}t_0^6\int_{\R^3}|u|^6\\[5mm]
&<\frac{1}{4}\int_{\R^3}(a|Du|^2+|u|^2)+\int_{\R^3}\left(\frac{1}{4}f(u)u-F(u)\right)+\frac{1}{12}\int_{\R^3}|u|^6\\[5mm]
&=\liminf\limits_{n\rightarrow\infty}\left[\frac{1}{4}\int_{\R^3}(a|Du_n|^2+|u_n|^2)+\int_{\R^3}\left(\frac{1}{4}f(u_n)u_n-F(u_n)\right)+\frac{1}{12}\int_{\R^3}|u_n|^6\right]\\[5mm]
&=\liminf\limits_{n\rightarrow\infty}(I(u_n)-\frac{1}{4}\langle I^\prime(u_n),u_n\rangle)\\[5mm]
&=c,\end{array}$$
which is impossible. Then $\int_{\R^3}|Du|^2=A^2=\lim\limits_{n\rightarrow\infty}\int_{\R^3}|Du_n|^2$ and $I^\prime(u)=0$.

(ii)~~~Since $p\in[2,4)$, by \eqref{3.2} and the H\"{o}lder and Yang inequalities, we have that
$$\begin{array}{ll}
c_\lambda+1+\|\tilde{u}_n\|&\geq I_{f_\lambda}(\tilde{u}_n)-\frac{1}{6}\langle I_{f_\lambda}^\prime(\tilde{u}_n),\tilde{u}_n\rangle\\[5mm]
&\geq \frac{1}{3}\|\tilde{u}_n\|^2+\frac{b}{12}|D\tilde{u}_n|_2^4-\lambda\frac{6-p}{6p}|f^+|_{p_*}S^{-\frac{p}{2}}|D\tilde{u}_n|_2^p\\[5mm]
&\geq\frac{1}{3}\|\tilde{u}_n\|^2-\frac{4-p}{4}\left(\frac{3p}{b}\right)^{\frac{p}{4-p}}\left(\lambda\frac{6-p}{6p}\frac{|f^+|_{p_*}}{S^{\frac{p}{2}}}\right)^{\frac{4}{4-p}},
\end{array}
$$
then $\{\tilde{u}_n\}$ is bounded in $H^1(\R^3)$.
\end{proof}

For all $\varepsilon>0$, we consider
\begin{equation}\label{3.6}
U_\varepsilon(x)=\frac{(3\varepsilon^2)^{\frac{1}{4}}}{(\varepsilon^2+|x|^2)^{\frac{1}{2}}},~~x\in\R^3,
\end{equation} which is a solution of the critical problem $-\Delta u=u^5$ in
$\R^3$ and
$\int_{\R^3}|DU_\varepsilon|^2=\int_{\R^3}|U_\varepsilon|^6=S^{\frac{3}{2}}$.
Let $\varphi$ be a smooth cut off function, i.e. $\varphi\in
C_0^\infty(\R^3)$, $0\leq\varphi\leq1$ and there exists $R>0$ satisfying supp$\varphi\subset
B_{2R}$, $\varphi|_{B_R}\equiv1$ and $|D\varphi|\leq\frac{2}{R}$.

Set $$w_\varepsilon=\varphi U_\varepsilon
~~~\hbox{and}~~~
v_\varepsilon=\frac{w_\varepsilon}{|w_\varepsilon|_6^6}.$$ By
\cite{bn}, we have
\begin{equation}\label{3.7}
|Dv_\varepsilon|_2^2=S+O(\varepsilon)
\end{equation}
and for any $s\in[2,6)$,
\begin{equation}\label{3.8}
|v_\varepsilon|_s^s=\left\{%
\begin{array}{ll}
    O(\varepsilon^{\frac{s}{2}}), & \hbox{$\hbox{if}~s\in[2,3)$}, \\
    O(\varepsilon^{\frac{3}{2}}|\hbox{log}~\varepsilon|), & \hbox{$\hbox{if}~s=3$},\\
    O(\varepsilon^{\frac{6-s}{2}}), & \hbox{$\hbox{if}~s\in(3,6)$}.\\
\end{array}%
\right.
\end{equation}

\begin{lem}\label{lem3.3}\ \
$c<c^*\triangleq\frac{abS^3}{4}+\frac{b^3S^6}{24}+\frac{(b^2S^4+4aS)^{\frac{3}{2}}}{24}$.
\end{lem}
\begin{proof}~~
For any $\varepsilon>0$, $t\geq0$, set
$$h_\varepsilon(t)\triangleq
I(tv_\varepsilon)=\frac{t^2}{2}\|v_\varepsilon\|^2+\frac{t^4}{4}b\left(\ds\int_{\R^3}|Dv_\varepsilon|^2\right)^2-\ds\int_{\R^3}F(tv_\varepsilon)-\frac{t^6}{6}.$$
By $(f_1)-(f_3)$, we easily see that $h_\varepsilon(t)$ has a unique
critical point $t_\varepsilon>0$ which corresponds to its maximum, i.e. $I(t_\varepsilon v_\varepsilon)=\max\limits_{t\geq0}
I(tv_\varepsilon)$. We claim that
$\{t_\varepsilon\}_{\varepsilon>0}$ is bounded from below by a
positive constant. Otherwise, there exists a sequence
$\{\varepsilon_n\}\subset\R^+$ satisfying
$\lim\limits_{n\rightarrow\infty}t_{\varepsilon_n}=0$ and
$I(t_{\varepsilon_n}v_{\varepsilon_n})=\max\limits_{t\geq0}
I(tv_{\varepsilon_n})$, then $0<\alpha\leq
c\leq\lim\limits_{n\rightarrow\infty}I(t_{\varepsilon_n}v_{\varepsilon_n})=0$,
which is impossible. So there exists $C>0$ independent of
$\varepsilon$ satisfying
\begin{equation}\label{3.9}
t_\varepsilon>C>0~~\hbox{for~all}~\varepsilon>0.
\end{equation}

Let
$$g(t)\triangleq\frac{t^2}{2}\|v_\varepsilon\|^2+\frac{t^4}{4}b\left(\ds\int_{\R^3}|Dv_\varepsilon|^2\right)^2-\frac{t^6}{6}.$$
We see that
\begin{equation}\label{3.10}
\begin{array}{ll}
\max\limits_{t\geq0}
g(t)&=\frac{b\|v_\varepsilon\|^2\left(\int_{\R^3}|Dv_\varepsilon|^2\right)^2}{4}+\frac{b^3\left(\int_{\R^3}|Dv_\varepsilon|^2\right)^6}{24}+\frac{\left[b^2\left(\int_{\R^3}|Dv_\varepsilon|^2\right)^4+4\|v_\varepsilon\|^2\right]^{\frac{3}{2}}}{24}\\[5mm]
&=\frac{abS^3}{4}+\frac{b^3S^6}{24}+\frac{(b^2S^4+4aS)^{\frac{3}{2}}}{24}+O(\varepsilon)
\end{array}
\end{equation}for $\varepsilon>0$ small enough.
By \eqref{3.7}-\eqref{3.10} and \eqref{3.26}, we have that
\begin{equation}\label{3.11}
\begin{array}{ll}
c\leq\max\limits_{t\geq0} I(tv_\varepsilon)&=I(t_\varepsilon
v_\varepsilon)\\[5mm]
&=\frac{t_\varepsilon^2}{2}\|v_\varepsilon\|^2+\frac{bt_\varepsilon^4}{4}\left(\ds\int_{\R^3}|Dv_\varepsilon|^2\right)^2-\ds\int_{\R^3}F(tv_\varepsilon)-\frac{t_\varepsilon^6}{6}\\[5mm]
&\leq
g(t_\varepsilon)-C(t_\varepsilon)^\mu\ds\int_{\R^3}|v_\varepsilon|^\mu\\[5mm]
&\leq \max\limits_{t\geq0} g(t)-C|v_\varepsilon|_\mu^\mu\\[5mm]
&=\frac{abS^3}{4}+\frac{b^3S^6}{24}+\frac{(b^2S^4+4aS)^{\frac{3}{2}}}{24}+O(\varepsilon)-O(\varepsilon^{\frac{6-\mu}{2}}).
\end{array}
\end{equation}
Since $\mu>4$, the conclusion follows from \eqref{3.11}
for $\varepsilon>0$ small enough.
\end{proof}

\begin{lem}\label{lem3.4}\ \
There exists $\lambda^*>0$ such that for all
$\lambda\in(0,\lambda^*)$,
$$c_\lambda<c^*-C_0\lambda^{\frac{4}{4-p}},$$where $c^*$ is given in
Lemmas 3.3 and $C_0=\frac{4-p}{4}\left(\frac{3p}{b}\right)^{\frac{p}{4-p}}\left(\frac{6-p}{6p}\frac{|f^+|_{p_*}}{S^{\frac{p}{2}}}\right)^{\frac{4}{4-p}}$.
\end{lem}
\begin{proof}~~
Since $\Sigma=\{x\in\R^3|~f^+(x)\neq0\}$ is a nonempty domain, we may assume that $0\in \Sigma$. Let $R_0>0$ be a constant such that $B_{R_0}(0)\subset\Sigma$. Following \cite{bn2}, let $\eta\in C^\infty_0(\Sigma)$ satisfy
$$\eta\geq0~~~~\hbox{and}~~~~\eta|_{B_{R_0}(0)}\equiv1.$$
We consider the following
function
$$w_{\varepsilon}(x)=\eta(x)\widetilde{U}_{\varepsilon}(x),~~x\in\R^3,$$
where
$\widetilde{U}_{\varepsilon}(x)=\frac{U_{\varepsilon}(x)}{(3\varepsilon^2)^{\frac{1}{4}}}=\frac{1}{(\varepsilon^2+|x|^2)^{\frac{1}{2}}}$. By \cite{bn}, we have that
$$\ds\int_{\R^3}|Dw_{\varepsilon}|^2=(3\varepsilon^2)^{-\frac{1}{2}}[B+O(\varepsilon)],~~~~\ds\int_{\R^3}|w_{\varepsilon}|^6=(3\varepsilon^2)^{-\frac{3}{2}}[A+O(\varepsilon^3)],$$
and
$\int_{\R^3}|w_{\varepsilon}|^2=(3\varepsilon^2)^{-\frac{1}{2}}O(\varepsilon)$,
where $B=\int_{\R^3}|DU_1|^2,$ $A=\int_{\R^3}|U_1|^6$ and
$S=BA^{-\frac{1}{3}}.$

Let
$$h(t)=\frac{t^2}{2}\|w_{\varepsilon}\|^2+\frac{t^4}{4}b\left(\ds\int_{\R^3}|Dw_{\varepsilon}|^2\right)^2-\frac{t^6}{6}\ds\int_{\R^3}|w_{\varepsilon}|^6,~~\forall~t\geq0.$$
Then \begin{equation}\label{3.12}
\begin{array}{ll}
\max\limits_{t\geq0}h(t)&=\frac{b\|w_{\varepsilon}\|^2\left(\int_{\R^3}|Dw_{\varepsilon}|^2\right)^2}{4\int_{\R^3}|w_{\varepsilon}|^6}+\frac{b^3\left(\int_{\R^3}|Dw_{\varepsilon}|^2\right)^6}{24\left(\int_{\R^3}|w_{\varepsilon}|^6\right)^2}+\frac{\left[b^2\left(\int_{\R^3}|Dw_{\varepsilon}|^2\right)^4+4\|w_{\varepsilon}\|^2\int_{\R^3}|w_{\varepsilon}|^6\right]^{\frac{3}{2}}}{24\left(\int_{\R^3}|w_{\varepsilon}|^6\right)^2}\\[5mm]
&=\frac{ab}{4}\frac{(B+O(\varepsilon))^3}{A+O(\varepsilon^3)}+\frac{b^3}{24}\frac{(B+O(\varepsilon))^6}{(A+O(\varepsilon^3))^2}+\frac{1}{24}\frac{[b^2(B+O(\varepsilon))^4+4a(B+O(\varepsilon))(A+O(\varepsilon^3))]^{\frac{3}{2}}}{(A+O(\varepsilon^3))^2}\\[5mm]
&=\frac{abB^3A^{-1}}{4}+\frac{b^3B^6A^{-2}}{24}+\frac{(b^2B^4A^{-\frac{4}{3}}+4aBA^{-\frac{1}{3}})^{\frac{3}{2}}}{24}+O(\varepsilon)+O(\varepsilon^3)\\[5mm]
&=\frac{abS^3}{4}+\frac{b^3S^6}{24}+\frac{(b^2S^4+4aS)^{\frac{3}{2}}}{24}+O(\varepsilon)+O(\varepsilon^3)
\end{array}
\end{equation} for $\varepsilon>0$ is small.

For $c^*$ given in Lemma 3.3 and $C_0>0$, we can choose $\lambda_1>0$ such
that for any $\lambda\in (0,\lambda_1)$,
$$c^*-C_0\lambda^{\frac{4}{4-p}}>0.$$
By $p\in[2,4)$, we see that
$\lim\limits_{t\rightarrow+\infty}I_{f_\lambda}(tw_{\varepsilon})=-\infty$
and
$$I_{f_\lambda}(tw_{\varepsilon})\leq\frac{t^2}{2}\|w_{\varepsilon}\|^2+\frac{bt^4}{4}\left(\ds\int_{\R^3}|Dw_{\varepsilon}|^2\right)^2+\frac{t^p}{p}\ds\int_{\R^3}f^-(x)|w_{\varepsilon}|^p~~\hbox{for~all~}t\geq0,~\lambda>0,$$
which implies that there exists $t_0>0$ small such that
\begin{equation}\label{3.13}
\sup_{0\leq t\leq
t_0}I_{f_{\lambda}}(tw_{\varepsilon})<c^*-C_0\lambda^{\frac{4}{4-p}}~~\hbox{for~all}~\lambda\in(0,\lambda_1).
\end{equation}
We next consider the case where $t>t_0$. By $f^-=0~~\hbox{in}~\Sigma~~\hbox{and}~~w_{\varepsilon}=0~~\hbox{in}~\R^3\backslash\Sigma,$
then we see that
\begin{equation}\label{3.14}
\ds\int_{\R^3}f^-(x)|w_{\varepsilon}|^p=0.
\end{equation}
Let $\varepsilon\leq R_0$, then
\begin{equation}\label{3.15}
\begin{array}{ll}
\ds\int_{\R^3}f^+(x)|w_{\varepsilon}|^p&=\ds\int_{\Sigma}f^+(x)|w_{\varepsilon}|^p\\[5mm]
&\geq\ds\int_{B_{R_0}(0)}f^+(x)\frac{1}{(\varepsilon^2+|x|^2)^{\frac{p}{2}}}\\[5mm]
&\geq\frac{1}{(2R_0^2)^{\frac{p}{2}}}\ds\int_{B_{R_0}(0)}f^+(x)\triangleq
C_*=C_*(R_0,f^+,\Sigma,p),
\end{array}
\end{equation}
where $f^+\in L^{p_*}(\R^3)$ implies that $f^+\in L^1_{loc}(\R^3)$.

By \eqref{3.12},\eqref{3.14},\eqref{3.15}, for all
$\varepsilon=\lambda^{\frac{4}{4-p}}\in (0,R_0)$ and $t>t_0$, we
have that
\begin{equation}\label{3.16}
\begin{array}{ll}
I_{f_\lambda}(tw_{\varepsilon})&=\frac{t^2}{2}\|w_{\varepsilon}\|^2+\frac{bt^4}{4}\left(\ds\int_{\R^3}|Dw_{\varepsilon}|^2\right)^2-\frac{t^p}{p}\ds\int_{\R^3}f_{\lambda}(x)|w_{\varepsilon}|^p-\frac{t^6}{6}\ds\int_{\R^3}|w_{\varepsilon}|^6\\[5mm]
&\triangleq
h(t)-\lambda\frac{t^p}{p}\ds\int_{\R^3}f^{+}(x)|w_{\varepsilon}|^p\\[5mm]
&\leq\frac{abS^3}{4}+\frac{b^3S^6}{24}+\frac{(b^2S^4+4aS)^{\frac{3}{2}}}{24}+O(\varepsilon)-\frac{t_0^pC_*}{p}\lambda\\[5mm]
&=c^*+O(\lambda^{\frac{4}{4-p}})-\frac{t_0^pC_*}{p}\lambda.
\end{array}
\end{equation}
As $p\in[2,4)$, $\frac{4}{4-p}\geq2$. Then there exists a
$\lambda_2>0$ small such that for all
$\lambda\in(0,\lambda_2)$, we have that
\begin{equation}\label{3.17}
O(\lambda^{\frac{4}{4-p}})-\frac{t_0^pC_*}{p}\lambda<-C_0\lambda^{\frac{4}{4-p}}.
\end{equation}
Set
$\lambda_3=\min\{\lambda_1,\lambda_2,R_0^{\frac{4-p}{4}}\}$
and $\varepsilon=\lambda^{\frac{4}{4-p}}$, then by \eqref{3.13}\eqref{3.16}\eqref{3.17}, for all $\lambda\in(0,\lambda_3)$, we
have that
$$\sup_{t\geq0}I_{f_\lambda}(tw_{\varepsilon})\leq
C^*-C_0\lambda^{\frac{4}{4-p}}.$$ If
$$\lambda^*\triangleq\min\{\lambda_0,\lambda_3\}~~\hbox{for}~p=2~~~~\hbox{or}~~~\lambda^*\triangleq\lambda_3~~\hbox{for}~p\in(2,4),$$
then by Lemma 3.1, for all $\lambda\in(0,\lambda^*)$, we have
$$c_\lambda\leq\sup_{t\geq0}I_{f_\lambda}(tw_{\varepsilon})\leq
c^*-C_0\lambda^{\frac{4}{4-p}}.$$
\end{proof}

We need the following three lemmas to prove Theorem 1.3.
\begin{lem}\label{lem3.5} (\cite{wi}, Lemma
1.21)\ \
Let $r>0$ and $2\leq q<2^*$. If $\{u_n\}$ is bounded in
$H^1(\R^N)$ and
$$\sup\limits_{y\in\R^N}\ds\int_{B_r(y)}|u_n|^q\rightarrow0,~~n\rightarrow+\infty,$$
then $u_n\rightarrow0$ in $L^s(\R^N)$ for $2<s<2^*$.
\end{lem}

\begin{lem}\label{lem3.6}\ \
For $t,~s>0$, the following system
$$\left\{%
\begin{array}{ll}
    f(t,s)\triangleq t-aS(t+s)^{\frac{1}{3}}=0, &  \\
    g(t,s)\triangleq s-bS^2(t+s)^{\frac{2}{3}}=0. &\\
\end{array}%
\right. $$ has a unique solution $(t_0,s_0)$. Moreover,
\begin{equation}\label{3.18}\left\{%
\begin{array}{ll}
    f(t,s)\geq0, &  \\
    g(t,s)\geq0, &\\
\end{array}%
\right.\Rightarrow~~
\left\{%
\begin{array}{ll}
    t\geq t_0, &  \\
    s\geq s_0. &\\
\end{array}%
\right.\end{equation}
\end{lem}
\begin{proof}~~
If $f(t_0,s_0)=g(t_0,s_0)=0$, then $t_0+s_0=\frac{t_0^3}{a^3S^3}$. It is enough to solve the following equation
$$
\left(\frac{t_0^3-a^3S^3t_0}{a^3S^3}\right)^3=b^3S^6\left(\frac{t_0^3}{a^3S^3}\right)^2.
$$
Then
 $$t_0=\frac{abS^3+a\sqrt{b^2S^6+4aS^3}}{2}$$ and
$$s_0=\frac{b^3S^6+2abS^3+b^2S^3\sqrt{b^2S^6+4aS^3}}{2}.$$
If $f(t,s)\geq0$ and $g(t,s)\geq0$, then
$$t+s\geq aS(t+s)^{\frac{1}{3}}+bS^2(t+s)^{\frac{2}{3}},$$
hence
$$t+s\geq t_0+s_0,$$ where we have used a fact that the function $h(l)\triangleq l-aSl^{\frac{1}{3}}-bS^2l^{\frac{2}{3}},~l>0$ has a unique zero point $l_0>0$ and $h(l)\geq0\Rightarrow l\geq l_0$.

Just suppose that $t<t_0$, then
$$f(t,s)=t-aS(t+s)^{\frac{1}{3}}<t_0-aS(t_0+s_0)^{\frac{1}{3}}=0,$$
which is impossible, so $t\geq t_0$. Similarly, $s\geq s_0$. The Lemma is proved.
\end{proof}

\begin{lem}\label{lem3.7}\ \
Let $\alpha\in(0,c^*)$ and $\{u_n\}\subset H^1(\R^3)$ be a bounded
$(PS)_\alpha$ sequence for $I$, then there exists a sequence
$\{y_n\}\subset \R^3$ and constants $R,~\sigma>0$ such that
$$\liminf\limits_{n\rightarrow\infty}\ds\int_{B_R(y_n)}|u_n|^2\geq\sigma>0.$$
\end{lem}
\begin{proof}~~
Just suppose that the conclusion does not hold, then by
Lemma 3.5, we see that $u_n\rightarrow0$ in $L^s(\R^3)$ for
$s\in(2,6)$, hence by $(f_1)(f_3)$, $\int_{\R^3}F(u_n)\rightarrow0$ and
$\int_{\R^3}f(u_n)u_n\rightarrow0$. Since $\{u_n\}$ is a bounded
$(PS)_\alpha$ sequence, we see that
\begin{equation}\label{3.19}
\alpha+o(1)=I(u_n)=\frac{1}{2}\|u_n\|^2+\frac{b}{4}\left(\ds\int_{\R^3}|Du_n|^2\right)^2-\frac{1}{6}\ds\int_{\R^3}|u_n|^6+o(1),
\end{equation}
\begin{equation}\label{3.20}
o(1)=\langle
I^\prime(u_n),u_n\rangle=\|u_n\|^2+b\left(\ds\int_{\R^3}|Du_n|^2\right)^2-\ds\int_{\R^3}|u_n|^6+o(1),
\end{equation}
where $o(1)\rightarrow0$ as $n\rightarrow\infty$. By \eqref{3.20},
we may assume that
$$\|u_n\|^2\rightarrow
l_1,~~~~b\left(\ds\int_{\R^3}|Du_n|^2\right)^2\rightarrow
l_2,~~~~\ds\int_{\R^3}|u_n|^6\rightarrow l_3.$$ Then by
\eqref{3.19}\eqref{3.20}, we have that
\begin{equation}\label{3.21}
\left\{%
\begin{array}{ll}
    l_1+l_2-l_3=0, &  \\
    \frac{1}{2}l_1+\frac{1}{4}l_2-\frac{1}{6}l_3=\alpha. &\\
\end{array}%
\right.
\end{equation}
So
\begin{equation}\label{3.22}
\alpha=\frac{1}{3}l_1+\frac{1}{12}l_2.
\end{equation}
Then $\alpha>0$ implies that
$l_1,l_2,l_3>0.$
By \eqref{3.2}, we have that
$$\ds\int_{\R^3}|Du_n|^2\geq
S\left(\ds\int_{\R^3}|u_n|^6\right)^{\frac{1}{3}}~~~~\hbox{and}~~~~b\left(\ds\int_{\R^3}|Du_n|^2\right)^2\geq
bS^2\left(\ds\int_{\R^3}|u_n|^6\right)^{\frac{2}{3}}.$$ Then
$$l_1\geq aS(l_1+l_2)^{\frac{1}{3}}~~\hbox{and}~~l_2\geq
bS^2(l_1+l_2)^{\frac{2}{3}}.$$ By Lemma 3.6, we have that
\begin{equation}\label{3.23}
\begin{array}{ll}
\alpha&=\frac{1}{3}l_1+\frac{1}{12}l_2\\[5mm]
&\geq\frac{1}{3}\frac{abS^3+a\sqrt{b^2S^6+4aS^3}}{2}+\frac{1}{12}
\frac{b^3S^6+2abS^3+b^2S^3\sqrt{b^2S^6+4aS^3}}{2}\\[5mm]
&=\frac{abS^3}{4}+\frac{b^3S^6}{24}+\frac{(b^2S^4+4aS)^{\frac{3}{2}}}{24}=c^*,
\end{array}
\end{equation}
which contradicts to $\alpha<c^*$. So the lemma is proved.
\end{proof}
\begin{lem}\label{3.8}(Lemma 1.32, \cite{wi})\ \
Let $\Omega$ be an open subset of $\R^N$  and let $\{u_n\}\subset L^p(\Omega),$ $1\leq p<\infty$. If $\{u_n\}$ is bounded in $L^p(\Omega)$ and $u_n\rightarrow u$ a.e. on $\Omega$, then
$$\lim\limits_{n\rightarrow\infty}(|u_n|_p^p-|u_n-u|_p^p)=|u|_p^p.$$
\end{lem}

We need the following compactness lemma to prove Theorem 1.4.

\begin{lem}\label{lem3.9}\ \
Let $\lambda^*$ be defined in Lemma 3.4, then for
$\lambda\in(0,\lambda^*),$ $I_{f_\lambda}$ satisfies $(PS)_{c_\lambda}$ condition.
\end{lem}
\begin{proof}~~
Suppose that $\{u_n\}$ is a $(PS)_{c_\lambda}$ sequence for
$I_{f_\lambda}$, by Lemma 3.2 (ii), we see that $\{u_n\}$ is bounded in $H^1(\R^3)$. Then there exists a $u\in
H^1(\R^3)$ and $A\in\R$ such that
\begin{equation}\label{3.4}
u_n\rightharpoonup u~~\hbox{in}~H^1(\R^3),
 \end{equation}
$$u_n(x)\rightarrow u(x)~~a.e.~\hbox{in}~\R^3$$
and
\begin{equation}\label{3.5}
\ds\int_{\R^3}|Du_n|^2\rightarrow A^2.\end{equation}
Then we have that
$$\ds\int_{\R^3}(aDuD\varphi+u\varphi)+bA^2\ds\int_{\R^3}DuD\varphi-\ds\int_{\R^3}f_\lambda(x)|u|^{p-1}u\varphi-\ds\int_{\R^3}u^5\varphi=0,~~\forall~\varphi\in H^1(\R^3)$$
i.e. $J_{f_\lambda}^\prime(u)=0$, where $$J_{f_\lambda}(u)=\frac{a+bA^2}{2}\ds\int_{\R^3}|Du|^2+\frac{1}{2}\ds\int_{\R^3}|u|^2-\frac{1}{p}\ds\int_{\R^3}f_{\lambda}(x)|u|^{p}-\frac{1}{6}\ds\int_{\R^3}|u|^6.$$
By \eqref{3.4}\eqref{3.5}, We see that $\{u_n\}$ is a bounded $(PS)_{c_\lambda+\frac{bA^4}{4}}$ sequence of $J_{f_\lambda}$ and
\begin{equation}\label{3.28}
\begin{array}{ll}
 J_{f_\lambda}(u)&=J_{f_\lambda}(u)-\frac{1}{6}\langle J_{f_\lambda}^\prime(u),u\rangle\\[5mm]
&\geq \frac{1}{3}\|u\|^2+\frac{b}{12}|Du|_2^4-\lambda\frac{6-p}{6p}|f^+|_{p_*}S^{-\frac{p}{2}}|Du|_2^p+\frac{bA^2}{4}\ds\int_{\R^3}|Du|^2\\[5mm]
&\geq-\frac{4-p}{4}\left(\frac{3p}{b}\right)^{\frac{p}{4-p}}\left(\lambda\frac{6-p}{6p}\frac{|f^+|_{p_*}}{S^{\frac{p}{2}}}\right)^{\frac{4}{4-p}}+\frac{bA^2}{4}\ds\int_{\R^3}|Du|^2.
\end{array}
\end{equation}

Set $v_n=u_n-u$, by \eqref{3.4}\eqref{3.5}, $f\in L^{p_*}(\R^3)$ and Lemma 3.8, we have that
$$\ds\int_{\R^3}|u_n|^2=\ds\int_{\R^3}|v_n|^2+\ds\int_{\R^3}|u|^2+o(1),~~~~\ds\int_{\R^3}|u_n|^6=\ds\int_{\R^3}|v_n|^6+\ds\int_{\R^3}|u|^6+o(1),$$
\begin{equation}\label{3.27}
A^2+o(1)=\ds\int_{\R^3}|Du_n|^2=\ds\int_{\R^3}|Dv_n|^2+\ds\int_{\R^3}|Du|^2+o(1)\end{equation}
and
$$\ds\int_{\R^3}f_\lambda(x)|u_n|^p=\ds\int_{\R^3}f_{\lambda}(x)|u|^p+o(1),$$
where $o(1)\rightarrow0$ as $n\rightarrow\infty$.
Then we have that
$$\begin{array}{ll}
o(1)&=\langle
J_{f_\lambda}^\prime(u_n),u_n\rangle\\[5mm]
&=\langle
J_{f_\lambda}^\prime(u),u\rangle+(a+bA^2)\ds\int_{\R^3}|Dv_n|^2+\ds\int_{\R^3}|v_n|^2-\ds\int_{\R^3}|v_n|^6+o(1)\\[5mm]
&=(a+bA^2)\ds\int_{\R^3}|Dv_n|^2+\ds\int_{\R^3}|v_n|^2-\ds\int_{\R^3}|v_n|^6+o(1)\\[5mm]
&=\|v_n\|^2+b\left(\ds\int_{\R^3}|Dv_n|^2\right)^2+b\ds\int_{\R^3}|Dv_n|^2\ds\int_{\R^3}|Du|^2-\ds\int_{\R^3}|v_n|^6+o(1).
\end{array}$$
Hence
\begin{equation}\label{3.24}
\|v_n\|^2+b\ds\int_{\R^3}|Dv_n|^2\ds\int_{\R^3}|Du|^2+b\left(\ds\int_{\R^3}|Dv_n|^2\right)^2-\ds\int_{\R^3}|v_n|^6=o(1).
\end{equation}
Up to a subsequence, we may assume that there exists $l_i\geq0$
$(i=1,\cdots,3)$ such that
$$\|v_n\|^2\rightarrow
l_1,~~~b\left(\ds\int_{\R^3}|Dv_n|^2\right)^2+b\ds\int_{\R^3}|Dv_n|^2\ds\int_{\R^3}|Du|^2\rightarrow
l_3,~~~\ds\int_{\R^3}|v_n|^6\rightarrow l_3,$$ then $l_1+l_2=l_3.$ If $l_1>0$,
then $l_3>0$. By \eqref{3.27}\eqref{3.24}, we see that
$$\begin{array}{ll}
J_{f_\lambda}(u_n)&=J_{f_\lambda}(u)+\frac{a+bA^2}{2}\ds\int_{\R^3}|Dv_n|^2+\frac{1}{2}\ds\int_{\R^3}|v_n|^2-\frac{1}{6}\ds\int_{\R^3}|v_n|^6+o(1)\\[5mm]
&=J_{f_\lambda}(u)+\frac{1}{2}\|v_n\|^2+\frac{b}{4}\left(\ds\int_{\R^3}|Dv_n|^2\right)^2+\frac{b}{4}\ds\int_{\R^3}|Dv_n|^2\ds\int_{\R^3}|Du|^2-\frac{1}{6}\ds\int_{\R^3}|v_n|^6\\[5mm]
&~~~~~~+\frac{bA^2}{4}\ds\int_{\R^3}|Dv_n|^2+o(1)\\[5mm]
&= J_{f_\lambda}(u)+\frac{1}{3}\|v_n\|^2+\frac{b}{12}\left[\left(\ds\int_{\R^3}|Dv_n|^2\right)^2+\ds\int_{\R^3}|Dv_n|^2\ds\int_{\R^3}|Du|^2\right]+\frac{bA^2}{4}\ds\int_{\R^3}|Dv_n|^2+o(1).
\end{array}
$$
Then let $n\rightarrow\infty$, we have that
$${c_\lambda+\frac{bA^4}{4}}\geq J_{f_\lambda}(u)+\frac{1}{3}l_1+\frac{1}{12}l_2+\frac{bA^2}{4}\lim\limits_{n\rightarrow\infty}\ds\int_{\R^3}|Dv_n|^2.$$
By \eqref{3.28}\eqref{3.27} and Lemma 3.6, similar to the proof of
\eqref{3.23}, we see that
$$\begin{array}{ll}
c_\lambda+\frac{bA^4}{4}&\geq J_{f_\lambda}(u)+\frac{1}{3}l_1+\frac{1}{12}l_2+\frac{bA^2}{4}\lim\limits_{n\rightarrow\infty}\ds\int_{\R^3}|Dv_n|^2\\[5mm]
&\geq-\frac{4-p}{4}\left(\frac{3p}{b}\right)^{\frac{p}{4-p}}\left(\lambda\frac{6-p}{6p}\frac{|f^+|_{p_*}}{S^{\frac{p}{2}}}\right)^{\frac{4}{4-p}}+c^*+\frac{bA^2}{4}\lim\limits_{n\rightarrow\infty}\left(\ds\int_{\R^3}|Du|^2+\ds\int_{\R^3}|Dv_n|^2\right)\\[5mm]
&\triangleq-C_0\lambda^{\frac{4}{4-p}}+c^*+\frac{bA^4}{4},
\end{array}$$
i.e.$$c_\lambda\geq -C_0\lambda^{\frac{4}{4-p}}+c^*,$$
which contradicts Lemma 3.4. So
$l_1=0$, i.e. $\|v_n\|^2\rightarrow0$, hence $u_n\rightarrow
u~~\hbox{in}~H^1(\R^3).$

\end{proof}

\section{Proof of Theorem 1.3 and 1.4}
\noindent $\textbf{Proof of Theorem 1.3}$\,\,\
\begin{proof}~~ We complete the proof in two steps.

\textbf{Step~1}~~~By Lemmas 3.1 and 3.2, there exists a bounded
$(PS)_c$ sequence $\{u_n\}\subset H^1(\R^3)$ for $I$. By Lemmas 3.3
and 3.7, there exists a sequence
$\{y_n\}\subset \R^3$ and constants $R,~\sigma>0$ such that
$\liminf\limits_{n\rightarrow\infty}\int_{B_R(y_n)}|u_n|^2\geq\sigma>0.$

Set $v_n(\cdot)=u_n(\cdot+y_n)$. Using the invariance of
$\R^3$ by translations, we see that $\{v_n\}$ is a bounded $(PS)_c$
sequence and
\begin{equation}\label{4.1}
\liminf_{n\rightarrow\infty}\ds\int_{B_R(0)}|v_n|^2\geq\sigma>0.
\end{equation}
By Lemma 3.2, there exists a $v\in H^1(\R^3)$ such that
$v_n\rightharpoonup v$ in $H^1(\R^3)$ and $I^\prime(v)=0$.
Furthermore, by \eqref{4.1}, we have that
$\int_{B_R(0)}|v|^2\geq\sigma>0.$
Therefore, $v$ is a nontrivial critical point for $I$. By Lemma 3.2, we have that
\begin{equation}\label{4.2}
\ds\int_{\R^3}|Dv_n|^2\rightarrow\ds\int_{\R^3}|Dv|^2.
\end{equation}
Then by \eqref{3.2},
\begin{equation}\label{4.3}
v_n\rightarrow v~~~~\hbox{in}~~L^6(\R^3).
\end{equation}
 So by $(f_1)(f_3)$, for any $\varepsilon>0$, there is $C_\varepsilon>0$ such that $|f(t)|\leq \varepsilon |t|+C_\varepsilon|t|^5$. Then by \eqref{4.3} and $v_n\rightharpoonup v$ in $H^1(\R^2)$, we see that
 \begin{equation}\label{4.5}
 \ds\int_{\R^3}|f(v_n)-f(v)||v_n-v|\leq\varepsilon(|v_n|_2+|v|_2)|v_n-v|_2+C_\varepsilon(|v_n|_6^5+|v|_6^5)|v_n-v|_6\rightarrow0.
\end{equation}
Then by \eqref{4.2}-\eqref{4.5}, $I^\prime(v_n)\rightarrow0$ in $H^{-1}(\R^3)$ and $I^\prime(v)=0$, we see that
$$v_n\rightarrow v~~~~\hbox{in}~~H^1(\R^3).$$
Therefore, $v$ is a nontrivial critical point of $I$ and $I(v)=c$.

\textbf{Step~2}~~~We next show that \eqref{1.1} has a positive
ground state solution in $H^1(\R^3)$.

Set $M=\{u\in H^1(\R^3)\backslash\{0\}|~I^\prime(u)=0\}$. We see
that $v\in M$, then $M\neq\varnothing$. For each $u\in M$, by
$(f_4)$,
$$I(u)=I(u)-\frac{1}{4}\langle
I^\prime(u),u\rangle=\frac{1}{4}\|u\|^2+\ds\int_{\R^3}H(u)+\frac{1}{12}\ds\int_{\R^3}|u|^6>0,
$$where $H(u)=\frac{1}{4}f(u)u-F(u).$

Let $m=\inf\limits_{u\in M}I(u)$, then
\begin{equation}\label{4.6}
0< m\leq I(v)=c<c^*.
\end{equation}
 Choosing a minimizing sequence
$\{u_n\}\subset M$ for $m$, i.e. $I(u_n)\rightarrow m$ and
$I^\prime(u_n)=0$, then $\{u_n\}$ is a $(PS)_m$ sequence for $I$ with $0<m<c^*$, then similar to the proof in
\textbf{Step~1}, there exists a nontrivial $u\in H^1(\R^3)$
satisfying $I^\prime(u)=0$ and $I(u)=m$, i.e. $u\in
H^1(\R^3)$ is a nontrivial ground state solution of \eqref{1.1}.

If we replace $I$ with the following functional
$$I^+(u)=\frac{1}{2}\ds\int_{\R^3}(a|Du|^2+|u|^2)+\frac{b}{4}\left(\ds\int_{\R^3}|Du|^2\right)^2-\ds\int_{\R^3}F(u)-\frac{1}{6}\ds\int_{\R^3}(u^+)^{6},
$$
where $u^\pm=\max\{\pm u,0\},$ then we see that all the calculations above can be repeated word by word. So there exists a nontrivial ground state critical point $u\in H^1(\R^3)$ of $I^+$. $\langle I^+(u),u^-\rangle=0$ implies that $\|u^-\|=0$, i.e. $u^-=0$, hence $u\geq0$ is a ground state solution of \eqref{1.1}. By using the strong
maximum principle and standard arguments, see e.g.
\cite{be,li,m,to,tr}, we see that $u(x)>0$, $\forall~x\in\R^3$.
Therefore, $u$ is a positive ground state solution of \eqref{1.1}
and the proof is completed.
\end{proof}

\noindent $\textbf{Proof of Theorem 1.4}$\,\,\
\begin{proof}~~
Theorem 1.4 is a direct consequence of Lemmas 3.1, 3.2, 3.4 and 3.9.
\end{proof}



 \end{document}